\documentclass[12pt]{amsart}
\usepackage[english]{babel}
\usepackage[latin1]{inputenc}
\usepackage{amstext}
\usepackage{amsmath}
\usepackage{amsfonts}
\usepackage{amssymb}
\usepackage{amsthm}
\usepackage[all]{xy}
\xyoption{arc}
\CompileMatrices
\usepackage{graphicx}

\newtheorem{thm}{Theorem}[section]

\newtheorem{lem}[thm]{Lemma}
\newtheorem{prop}[thm]{Proposition}
\newtheorem{dfn}[thm]{Definition}

\newcommand{\Ext}{\operatorname{Ext}}
\newcommand{\rk}{\operatorname{rk}}
\newcommand{\ch}{\operatorname{ch}}
\newcommand{\td}{\operatorname{td}}

 \selectfont

\numberwithin{equation}{section}

\begin{document}

\title[Gieseker Harder-Narasimhan filtration for principal bundles]{On
the Gieseker Harder-Narasimhan filtration for principal bundles}

\author[I. Biswas]{Indranil Biswas}

\address{School of Mathematics, Tata Institute of Fundamental Research, Homi Bhabha Road, Mumbai 400005, India}

\email{indranil@math.tifr.res.in}

\author[A. Zamora]{Alfonso Zamora}

\address{Departamento de Matem\'atica, Instituto Superior T\'ecnico, Av. Rovisco Pais
1049-001 Lisboa, Portugal}

\email{alfonsozamora@tecnico.ulisboa.pt}



\begin{abstract}
We give an example of an orthogonal bundle where the Harder-Narasimhan 
filtration, with respect to Gieseker semistability, of its underlying vector bundle
does not correspond to any parabolic reduction of the orthogonal bundle. A similar
example is given for the symplectic case.
\end{abstract}

\maketitle

\section{Introduction}

Mumford and Takemoto (c.f. \cite{GIT} and \cite{Ta}) constructed a moduli space for semistable holomorphic 
vector bundles over Riemann surfaces, by defining a vector bundle $E$ to be semistable 
if for every non trivial proper subbundle $F\subset E$
we have
$$\frac{\deg F}{\rk F}\leq \frac{\deg E}{\rk E}\; .$$
Later on, Gieseker and Maruyama (c.f. \cite{Gi} and \cite{Ma}) extended the construction to higher dimensional varieties, by 
modifying the definition of stability and using Hilbert polynomials instead of degrees. 

In \cite{HN}, Harder and Narasimhan proved that a holomorphic vector bundle over a Riemann surface which is not stable admits a unique canonical filtration
$$0\subset E_{1} \subset E_{2} \subset \cdots \subset E_{t} \subset
E_{t+1}=E\; ,$$ which verifies that the quotients $E^{i}:=E_{i}/E_{i-1}$ are semistable and the slopes of the quotients
are decreasing:
   $$\frac{\deg E^{1}}{\rk E^{1}}>\frac{\deg E^{2}}{\rk E^{2}}>\cdots>\frac{\deg E^{t+1}}{\rk E^{t+1}}\; .$$
This fact can be extended to torsion free sheaves over projective varieties by using Gieseker stability (c.f. \cite{Gi} and \cite[Theorem 1.3.6]{HL}).

Ramanathan, in \cite{Ra}, constructed a moduli space for semistable principal $G$-bundle over Riemann surfaces, 
where $G$ is a connected reductive complex algebraic group, by declaring
that a $G$-bundle $E\rightarrow X$ is semistable if for any reduction of structure group $\sigma :X\rightarrow E/P$
to any maximal parabolic subgroup $P\subset G$, we have 
$$\deg \sigma^{\ast}(T_{(E/P)/X})\geq 0$$
where $T_{(E/P)/X}\rightarrow E/P$ is the tangent bundle along the fibers of the projection
$E/P\rightarrow X$.

To generalize Ramanathan's construction to higher dimension, a
Gieseker-like definition of stability has to be given for principal
bundles. It turns out that this definition depends on the choice of
a representation $\rho:G\to {\rm SL}(V)$. 
For the classical groups, we can choose the standard representation
(c.f. \cite{GS1}). For a connected reductive group, we can take the 
adjoint representation into the semisimple part of the Lie algebra
(c.f. \cite{GS2,GS3}), and if $G$ is semisimple we can take any faithful 
representation (c.f. \cite{Sch1,Sch2}). See \cite{GLSS1} and \cite{GLSS2}
for other cases.

A one parameter subgroup $\lambda:\mathbb{C}^* \to G$ of $G$ defines a
parabolic subgroup
$$
P(\lambda)=\{ g\in G \,\mid\, \lim_{z\to \infty} \lambda(t)g\lambda(t)^{-1}
~\text{exists in $G$} \}
$$
We remark that all parabolic subgroups are of this form.

Using the representation $\rho:G \to {\rm SL}(V)$ chosen above,
the
one parameter subgroup $\lambda$ gives a filtration of $V$ as
follows. Let 
$$
\gamma_1 < \gamma_2 <\cdots < \gamma_l
$$
be the different weights of the action of $\lambda$ on $V$. Let
$V^{i}$ be the subspace where $\lambda$ acts as $t^{\gamma_i}$, and
let $V_i=\bigoplus_{j=1}^{i}$. Therefore, we obtain a filtration
\begin{equation}
\label{filtV}
0\subset V_1 \subset V_2 \subset \cdots \subset V_l=V  
\end{equation}
This filtration is preserved by the action of the parabolic group $P(\lambda)$
on $V$ given by the representation.

We say that an open subset $U\subset X$ is big if the dimension of
its complement has codimension at least 2.
Consider a reduction of structure group of the principal $G$-bundle $E$ to $P(\lambda)$ over a
big open set $\iota:U\subset X$. 
This reduction, together with the filtration of $V$ \eqref{filtV}
produces a filtration of vector bundles on $U$
$$
0\subset F^U_1 \subset F^U_2 \subset \cdots \subset F^U_l = E(V)|_U
$$ 
Extending each vector bundle to a torsion free sheaf on $X$ as
$F_i = \iota_* F^U_i \cap E(V)$, we obtain a filtration by torsion
free sheaves on $X$
\begin{equation}
  \label{filtF}
0\subset F_1 \subset F_2 \subset \cdots \subset F_l = E(V)  
\end{equation}

We say that $E$ is semistable if for any one parameter subgroup 
which factors through the derived subgroup $[G,G]$ of $G$
$$
\rho:\mathbb{C}^* \to [G,G] \to G
$$
and for any reduction of structure group
of $E$ to the parabolic subgroup $P(\lambda)$ over a big open 
set $U\subset X$, the associated filtration of torsion free sheaves 
\eqref{filtF} satisfies
$$
\sum_{i=1}^{l-1} (\gamma_{i+1}-\gamma_{i})
\big(\rk F\cdot P_{F_i}(m) -\rk F_i\cdot P_F(m)\big) \leq 0
$$
for $m\gg 0$. We say that it is stable if the inequality is strict.

In the case of principal $G$-bundles over Riemann surfaces, this
notion of stability coincides with the one in \cite{Ra} 
(c.f. \cite[Corollary 5.8]{GS3}, \cite[Proposition 5.4]{Sch2} , 
\cite[Lemma 3.2.3]{GLSS1}), \cite[Lemma 2.5.4]{GLSS2})
and therefore does not depend on the choice the representation $\rho$.

In the case of the adjoint representation into the semisimple part of the Lie algebra,
the filtrations (\ref{filtF}) that we obtain can be simplified as follows (\cite{GS2,GS3}).

Let $E$ be a principal $G$-bundle over $X$, let $F=E(\frak{g}')$ be the 
vector bundle associated to the adjoint representation in the semisimple
part $\frak{g}'$ of the Lie algebra $\frak{g}=\frak{g}'\bigoplus Z(\frak{g})$.
Let
$$[\; , \; ]: F\otimes F\rightarrow F$$
be the Lie algebra structure and let
$$\kappa: F\otimes F\rightarrow \mathcal{O}_{X}$$
be the Killing form, inducing an isomorphism $F\cong F^{\vee}$ and assigning an orthogonal
$$F'{}^{\bot}= \ker(F\hookrightarrow F^{\vee}\twoheadrightarrow F'{}^{\vee})$$
to each subsheaf $F'\subset F$. An orthogonal algebra filtration of $F$ is a filtration 
\begin{equation}
 \label{orthogonalfiltration}
0\subset F_{-t}\subset F_{-t+1}\subset \cdots \subset F_{0}\subset \cdots \subset F_{t-1}\subset F_{t}=F
\end{equation}
verifying $F_{i}^{\bot}=F_{-i-1}$ and $[F_{i},F_{j}]\subset F_{i+j}$. Define
$$P_{F_{\bullet}}:=\sum_{i} (\rk F \cdot P_{F_{i}}-\rk F_{i}\cdot P_{F})$$
to be the Hilbert polynomial of a filtration and declare $E$ to be semistable if 
for every orthogonal algebra filtration $F_{\bullet}\subset F$ we have
$$P_{F_{\bullet}}\leq 0\; .$$

The case of the orthogonal group with the standard representation will be described
in more detail in Section \ref{sec:constructing}.

We will now describe in more detail the case of the adjoint representation. 

Ramanathan states in \cite{Ra} that every unstable principal $G$-bundle admits a unique canonical reduction $\sigma$ 
to a parabolic $P\subset G$ satisfying the following two conditions:
\begin{enumerate}
 \item \label{can1} 
 for every non-trivial character $\chi$ of $P$ given by a nonnegative linear combination of simple roots 
 (with respect to some fixed Borel subgroup of $P$), the associated line bundle $\chi_{\ast}\sigma^{\ast}E$ has positive degree,
 \item \label{can2}
for the Levi quotient $P\twoheadrightarrow L$, the associated principal $L$-bundle is
semistable.
\end{enumerate}
In \cite{AB}, Atiyah and Bott showed that the Harder-Narasimhan filtration in 
for the adjoint vector bundle is indeed an orthogonal algebra filtration
as in (\ref{orthogonalfiltration}) and the middle term $E_{0}$ gives a reduction of
structure group of the principal $G$-bundle to a parabolic subgroup. In particular, when 
$G={\rm GL}(n,\mathbb{C})$, the canonical reduction corresponds to the
Harder-Narasimhan filtration of the associated vector bundle of rank $n$. In \cite{Be}
and \cite{BH} the assertion of Ramanathan is proved and
also shown that the reduction in \cite{AB} coincides with the one in \cite{Ra}.
Finally, \cite{AAB} generalizes the notion of the canonical reduction to principal 
$G$-bundles over compact K\"ahler manifolds $X$ by considering reductions to 
parabolics over a big open set $U\subset X$ which satisfy properties (\ref{can1}) 
and (\ref{can2}); the reduction in \cite{AAB} is constructed by following the
method in \cite{AB}.

Given an orthogonal or symplectic bundle over a projective variety $X$, we can construct the Gieseker Harder-Narasimhan filtration of its underlying 
vector bundle. On the other hand, we can construct the canonical reduction of the
principal bundle. When $X$ is of complex dimension $1$, both notions coincide,
however they differ for higher dimensional varieties. Here we
show an example of this. 

In \cite{GSZ}, the authors explore the connections between the maximal 
$1$-parameter subgroup giving maximal unstability from the GIT point of view in a 
GIT construction of a moduli space (c.f. \cite{Ke}), and the Harder-Narasimhan 
filtration. These ideas give a method to construct the Harder-Narasimhan filtration 
in cases where we do not know it a priori. The paper \cite{Za1} applies this to 
rank $2$ tensors and the method does not work in more general situations (c.f. 
\cite[Section 2.5]{Za2}). It is natural to ask whether there exists a Harder-Narasimhan
filtration for 
principal bundles coming from the construction of the moduli space in \cite{GS3} 
through tensors, where a Gieseker type stability is used. The present work was
motivated by this question.

\section{Definitions}

Let $X$ be a polarized smooth complex projective variety of dimension $n$ and let 
$\mathcal{O}_{X}(1)$ denote the polarization. Denote by $\rk E$ and $\deg E$ the rank 
and degree of a torsion free coherent sheaf $E$. Recall that we define the 
\emph{degree} of a vector bundle $E$ over the polarized variety
$(X,\mathcal{O}_{X}(1))$ as $$\deg E=\int_{X} c_{1}(E)\cdot 
c_{1}(\mathcal{O}_{X}(1))^{n-1}=\int_{X} c_{1}(E)\cdot H^{n-1}\; ,$$ where $H$ is a 
hyperplane divisor of the polarizing line bundle.

\begin{dfn}
A torsion free coherent sheaf $E$ over a smooth projective variety
is \emph{Mumford-Takemoto semistable} if for every nontrivial subsheaf
$F\subset E$, with $\rk F < \rk E$,
we have
$$\frac{\deg F}{\rk F}\leq \frac{\deg E}{\rk E}\; .$$
Moreover, if the above inequality is strict, we say $E$ is \emph{Mumford-Takemoto
stable}. We say that $F$ \emph{Mumford-Takemoto destabilizes}
$E$ if 
$$\frac{\deg F}{\rk F}>\frac{\deg E}{\rk E}\; .$$
\end{dfn}

We call $\deg E/\rk E$ the {\it Mumford-Takemoto slope} of $E$. 
Let $\chi(E)$ be the Euler characteristic of a sheaf $E$. Given a torsionfree coherent sheaf $E$ over a polarized variety $(X,\mathcal{O}_{X}(1))$, we define the \emph{Hilbert polynomial} of $E$
as 
$$P_{E}(m)=\chi(E\otimes \mathcal{O}_{X}(m))\; ,$$
which is a polynomial on $m$. Given two polynomials $P(m)$ and $Q(m)$ we say that $P(m)\leq Q(m)$ if the inequality holds for $m\gg 0$. 

\begin{dfn}
A torsion free coherent sheaf $E$ over a smooth projective variety is
\emph{Gieseker semistable} if for every nontrivial subsheaf $F\subset E$,
with $\rk F < \rk E$, we have
$$\frac{P_{F}(m)}{\rk F}\leq \frac{P_{E}(m)}{\rk E}\; .$$
Moreover, if the above inequality is strict, we say $E$ is \emph{Gieseker stable}. We say that $F$ \emph{Gieseker destabilizes}
$E$ if 
$$\frac{P_{F}(m)}{\rk F}>\frac{P_{E}(m)}{\rk E}\; .$$
\end{dfn}

The above polynomial $P_{E}(m)/\rk E$ is called the {\it Gieseker slope} of $E$. 

Note that $F$ Mumford-Takemoto destabilizes $E$ implies $F$ Gieseker destabilizes $E$ but 
not the other way around.

Given a torsion free sheaf $E$, there exists a unique filtration
\begin{equation}
\label{HNeq}
0\subset E_{1} \subset E_{2} \subset \cdots \subset E_{t} \subset
E_{t+1}=E\; ,
\end{equation}
which satisfies the following properties, where
$E^{i}:=E_{i}/E_{i-1}$:
 \begin{enumerate}
   \item Every $E^{i}$ is Gieseker semistable
   \item The Hilbert polynomials verify
   $$\frac{P_{E^{1}}(m)}{\rk E^{1}}>\frac{P_{E^{2}}(m)}{\rk E^{2}}>\ldots>\frac{P_{E^{t+1}}(m)}{\rk E^{t+1}}$$
 \end{enumerate}
This filtration is called the \emph{Harder-Narasimhan filtration} of
$E$ (c.f. \cite{HN} or \cite[Theorem 1.3.6]{HL}).

\section{A Gieseker unstable bundle which is Mumford-Takemoto semistable}

Let $X$ be a $K3$ surface, which means that $X$ is a smooth complex projective 
surface with irregularity $q(X)=h^{1}(\mathcal{O}_{X})=0$ and trivial canonical 
bundle. Since $X$ is K\"ahler, there exists a K\"ahler-Einstein metric on $X$ (c.f. 
\cite{Ya}), hence $TX$ is polystable (c.f. \cite{Ko}). Moreover, $TX$ is 
indecomposable, therefore $TX$ is Mumford-Takemoto stable. The first Chern class of $TX$
vanishes, so $\deg (TX)=0$.

We consider the sheaves $TX$ and $\mathcal{O}_{X}$ and calculate their Hilbert polynomials. 

Let us calculate $\chi(TX\bigotimes \mathcal{O}_{X}(m))$ by using the
Hirzebruch-Riemann-Roch theorem:
$$\chi(TX\otimes \mathcal{O}_{X}(m))=\deg(\ch(TX\otimes \mathcal{O}_{X}(m))\cdot \td(TX))_{2}\; ,$$
where the subscript $2$ refers to the component of the cohomology in
$H^4(X,{\mathbb Q})$.

The Chern character of a vector bundle is (c.f. \cite[Appendiz A.4]{Ha}) 
$\ch(E)=\sum_{i=1}^{r}e^{a_{i}}$, with $c_{t}(E)=\Pi_{i=1}^{r}(1+a_{i}t)$ being the Chern 
polynomial, while $\td(TX)$ denotes the Todd character of the tangent bundle. Given 
that $\ch(E\bigotimes F)=\ch(E)\cdot \ch(F)$, and $\mathcal{O}_{X}(m)$ is a line 
bundle,
$$\ch(\mathcal{O}_{X}(m))=e^{[\mathcal{O}_{X}(m)]}=e^{(mH)}=1+mH+\frac{(mH)^{2}}{2!}+\frac{(mH)^{3}}{3!}+\cdots 
\; ,$$ where $H$ is the class of a hyperplane divisor of $\mathcal{O}_{X}(1)$. 
Therefore, in our case, $$\chi(TX\otimes \mathcal{O}_{X}(m))=\deg(\ch(TX\otimes 
\mathcal{O}_{X}(m))\cdot \td(TX))_{2}= \deg(\ch(TX)\cdot \td(TX)\cdot 
\ch(\mathcal{O}_{X}(m)))_{2}\; .$$

The Chern and Todd classes of a vector bundle $E$ are given by 
$$\ch(E)=\rk E+c_{1}+\frac{1}{2}(c_{1}^{2}-2c_{2})+\ldots$$
$$\td(E)=1+\frac{1}{2}c_{1}+\frac{1}{12}(c_{1}^{2}+c_{2})+\ldots\; ,$$
then, in our case, 
$$\ch(TX)=\rk (TX)+c_{1}(TX)+\frac{1}{2}(c_{1}(TX)^{2}-2c_{2}(TX))=2-c_{2}(TX)$$
$$\td(TX)=1+\frac{1}{2}c_{1}(TX)+\frac{1}{12}(c_{1}(TX)^{2}+c_{2}(TX))=1+\frac{c_{2}(TX)}{12}$$
$$\ch(\mathcal{O}_{X}(m))=1+mH+\frac{m^{2}H^{2}}{2}\; .$$
Hence we have, 
$$\chi(TX\otimes \mathcal{O}_{X}(m))=$$
$$[(2-c_{2}(TX))\cdot (1+\frac{1}{12}c_{2}(TX))\cdot (1+mH+\frac{m^{2}H^{2}}{2})]_{2}=$$
$$[(2-\frac{5}{6}c_{2}(TX))\cdot (1+mH+\frac{m^{2}H^{2}}{2})]_{2}=$$
$$m^{2}H^{2}-\frac{5}{6}c_{2}(TX)\; .$$

Using the same arguments, let us calculate $\chi(\mathcal{O}_{X}(m))$,
$$\chi(\mathcal{O}_{X}(m))=\deg(\ch(\mathcal{O}_{X}(m))\cdot \td(TX))_{2}=$$
$$[(1+mH+\frac{m^{2}H^{2}}{2})\cdot (1+\frac{c_{2}(TX)}{12})=$$
$$\frac{m^{2}H^{2}}{2}+\frac{c_{2}(TX)}{12}\; .$$

The Euler characteristic of $\mathcal{O}_{X}$ is also given by
$$\chi(\mathcal{O}_{X})=h^{0}(\mathcal{O}_{X})-h^{1}(\mathcal{O}_{X})+h^{2}(\mathcal{O}_{X})=$$
$$2h^{0}(\mathcal{O}_{X})=2$$
because on a $K3$ surface we have $h^{1}(\mathcal{O}_{X})=0$ and, by Serre duality and triviality of the canonical bundle, 
$h^{2}(\mathcal{O}_{X})=h^{0}(\mathcal{O}_{X})$. On the other
hand, Hirzebruch-Riemann-Roch theorem gives
$$\chi(\mathcal{O}_{X})=\frac{1}{12}c_{2}(TX)\; ,$$
from which we get $c_{2}(TX)=24$. 

Therefore, we can write the Hilbert polynomials of $TX\bigotimes \mathcal{O}_{X}$
and $\mathcal{O}_{X}$:
$$P_{TX}(m)=\chi(TX\otimes \mathcal{O}_{X}(m))=$$
$$m^{2}H^{2}-\frac{5}{6}c_{2}(TX)=m^{2}H^{2}-20\; ,$$
$$P_{\mathcal{O}_{X}}(m)=\chi(\mathcal{O}_{X}(m))=$$
$$\frac{m^{2}H^{2}}{2}+\frac{c_{2}(TX)}{12}=\frac{m^{2}H^{2}}{2}+2\; .$$

\begin{prop}\label{prop1}
The Gieseker slope of the line bundle $\mathcal{O}_{X}$ is strictly bigger
than the Gieseker slope of $TX$. 
\end{prop}

\begin{proof}
{}From the above calculations we have
$$\frac{P_{\mathcal{O}_{X}}(m)}{\rk 
\mathcal{O}_{X}}=\frac{m^{2}H^{2}}{2}+2> \frac{P_{TX}(m)}{\rk 
TX}=\frac{m^{2}H^{2}-20}{2}=m^{2}H^{2}-10\; ,$$ proving the proposition.
\end{proof}

\section{Extensions and Harder-Narasimhan filtration}

First, we construct a stable bundle as an extension of $\mathcal{O}_{X}$ by $TX$. 

\begin{prop}\label{prop2}
Let $V$ be a nontrivial extension 
$$0\rightarrow TX\rightarrow V\rightarrow \mathcal{O}_{X}\rightarrow 0\; .$$ 
Then $V$ is Gieseker semistable.
\end{prop}

\begin{proof}
First, we see that there exists such a nontrivial extension. Indeed, 
$$\dim \Ext^{1}(\mathcal{O}_{X}, TX)=h^{1}(\mathcal{O}_{X}^{\vee}\otimes TX)=h^{1}(TX)\; ,$$
and
$$\chi(TX)=h^{0}(TX)-h^{1}(TX)+h^{2}(TX)=-20$$
as calculated before. Hence $h^{1}(TX)>0$, so and there exist nontrivial extensions. 

Suppose $V$ is not Gieseker semistable. Let $W\subset V$ be the maximal Gieseker
destabilizer (first nonzero term of the Harder-Narasimhan filtration (see
\eqref{HNeq})) of $V$. Consider
the short exact sequence
$$
0 \rightarrow W_0:= W\cap TX \rightarrow W\rightarrow W_1:= W/W_0 \rightarrow 0\, .
$$
Note that $W_1\subset V/TX = \mathcal{O}_{X}$.

First assume that $W_1 =0$. This implies that $W\subset TX$. But $TX$ is Gieseker
stable and the Gieseker slope of $TX$ is strictly smaller than the
Gieseker slope of $\mathcal{O}_{X}$ (Proposition \ref{prop1}), so the
Gieseker slope of $W$ is not bigger than the Gieseker slope of $V$. This contradicts
the assumption that $W$ is destabilizes $V$.

Now assume that $W_1 \not= 0$. So $\rk W_1 > 0$. Since $\rk W < 3$, this implies
that $\rk W_0 < 2$. As $TX$ is Gieseker stable (because it is Mumford-Takemoto stable), if $W_0\neq 0$ the
Gieseker slope of $W_0$ is strictly smaller than the Gieseker slope of $TX$. Also,
The Gieseker slope of $W_1$ is not bigger than the Gieseker slope of $\mathcal{O}_{X}$.
Therefore, the Gieseker slope of $W$ is not bigger than the Gieseker slope of $V$. This
again contradicts the assumption that $W$ is destabilizes $V$. Therefore, $W_0=0$ and $W$ would split the extension,
hence, $V$ is semistable.
\end{proof}

Let $V$ be a non split extension 
$$0\rightarrow TX\rightarrow V\rightarrow \mathcal{O}_{X}\rightarrow 0\; .$$
Dualizing the above sequence we obtain 
$$0\rightarrow \mathcal{O}_{X}\rightarrow V^{\vee}\rightarrow TX\rightarrow 0$$
because $\Omega^{1}(X)=(TX)^{\vee}\simeq TX$ with the isomorphism given by a
trivialization of the canonical line bundle on the $K3$ surface $X$. 

\begin{lem}
The above vector bundle $V^{\vee}$ is not Gieseker semistable.
\end{lem}

\begin{proof}
{}From Proposition \ref{prop1} it follows that the subbundle $\mathcal{O}_{X}$ of
$V^{\vee}$ destabilizes it.
\end{proof}

\begin{prop}
\label{HN}
The Gieseker Harder-Narasimhan filtration of $V\bigoplus V^{\vee}$ is
\begin{equation}\label{f1}
0\subset \mathcal{O}_{X}\subset V\oplus\mathcal{O}_{X}\subset V\oplus V^{\vee}\; .
\end{equation}
\end{prop}

\begin{proof}
Let us check that the filtration in \eqref{f1} satisfies the two properties which
characterize the Harder-Narasimhan filtration in \eqref{HNeq}. The successive
quotients for the filtration in \eqref{f1} are: $\mathcal{O}_{X}$, 
$(V\bigoplus \mathcal{O}_{X})/\mathcal{O}_{X}=V$, and $(V\bigoplus V^{\vee})/(V\bigoplus
\mathcal{O}_{X})=TX$. All these three vector bundles $\mathcal{O}_{X}$,
$V$ and $TX$ are Gieseker semistable; $V$ is semistable by Proposition \ref{prop2}
and $TX$ is semistable because it is Mumford-Takemoto stable.

Concerning the Hilbert polynomials of the above three vector bundles, 
$$\frac{P_{E^{1}}(m)}{\rk E^{1}}=\frac{P_{\mathcal{O}_{X}}(m)}{\rk \mathcal{O}_{X}}=\frac{m^{2}H^{2}}{2}+2$$
$$\frac{P_{E^{2}}(m)}{\rk E^{2}}=\frac{P_{V}(m)}{\rk V}=\frac{m^{2}H^{2}}{2}-6$$
$$\frac{P_{E^{3}}(m)}{\rk E^{3}}=\frac{P_{TX}(m)}{\rk TX}=\frac{m^{2}H^{2}}{2}-10\; ,$$
which are, clearly, decreasing. 
\end{proof}

\section{Constructing a principal bundle}
\label{sec:constructing}

Consider the vector bundle $V\bigoplus V^{\vee}$ in Proposition \ref{HN} which is self
dual. The duality gives an orthogonal structure on $V\bigoplus V^{\vee}$, i.e., a symmetric nondegenerate
homomorphism
$$\varphi: (V\oplus V^{\vee})\otimes (V\oplus V^{\vee})\rightarrow \mathcal{O}_{X}\, ;$$
nondegeneracy means that $\varphi$ induces an isomorphism between $(V\bigoplus V^{\vee})$ and its dual. 
In other words, we get a reduction of the structure group of the
${\rm GL}(6,\mathbb{C})$-bundle to ${\rm O}(6,\mathbb{C})$.

Given $(E,\varphi)$ an orthogonal bundle, we say that $F\subset E$ is 
\emph{isotropic} if $\varphi|_{F\otimes F}=0$. For a subsheaf $F\subset E$, by 
using $\varphi$ we can associate the \emph{annihilator} subsheaf $$F^{\bot}:=\ker 
(E\stackrel{\varphi}\rightarrow E^{\vee}\rightarrow F^{\vee})\; .$$

If we apply the general definition of stability in the introduction to the case
of the orthogonal bundle and the standard representation, we obtain the following
definition.

\begin{dfn}[{\cite[Definition 5.4]{GS1}}]
An orthogonal bundle $(E,\varphi)$ is called \emph{semistable} (as an orthogonal
bundle) if for every isotropic subbundle $F\subset E$, 
$$P_{F}(m)+P_{F^{\bot}}(m)\leq P_{E}(m)\; .$$
We say that $(E,\varphi)$ is \emph{stable} if the above inequality is strict. 
\end{dfn}

\begin{lem}
The Harder-Narasimhan filtration of $V\bigoplus V^{\vee}$ in Proposition \ref{HN}
does not correspond to any parabolic reduction of the principal ${\rm O}(6,
\mathbb{C})$-bundle.
\end{lem}

\begin{proof}
The annihilator of the subbundle ${\mathcal O}_X\subset V\bigoplus V^{\vee}$ in
\eqref{f1} is $TX\bigoplus V^{\vee}$. The lemma follows immediately from this.
\end{proof}

\subsection{Symplectic case}

The natural duality pairing between $V$ and $V^{\vee}$ also gives a
symplectic structure on $V\bigoplus V^{\vee}$. The annihilator of the subbundle
${\mathcal O}_X\subset V\bigoplus V^{\vee}$ in \eqref{f1} for the
symplectic structure on $V\bigoplus V^{\vee}$ is again $TX\bigoplus V^{\vee}$.
Therefore, the Harder-Narasimhan filtration of $V\bigoplus V^{\vee}$ in Proposition
\ref{HN} does not correspond to any parabolic reduction of this principal ${\rm Sp}(6,
\mathbb{C})$-bundle.

\section*{Acknowledgements}

We thank Tom\'as G\'omez for helpful comments. The first-named author thanks the
Instituto Superior  T\'ecnico for hospitality whilst the work was carried out. The
second-named author is supported by project \textquotedblleft Comunidade Portuguesa de Geometr\'ia Algebrica\textquotedblright
 RD0302 funded by Portuguese FCT and project MTM2010-17389 granted by Spanish Ministerio
de Econom\'ia y Competitividad, and the first-named author is supported by J. C. Bose
Fellowship.

\end{document}